\documentclass[11pt,reqno]{amsart}
\usepackage{amssymb,latexsym}
\usepackage{commath}
\usepackage{graphicx}
\usepackage{url}		%does nice formatting of URLs

\calclayout

\makeatletter
\newcommand{\monthyear}[1]{%
  \def\@monthyear{\uppercase{#1}}}
\newcommand{\volnumber}[1]{%
  \def\@volnumber{\uppercase{#1}}}
\AtBeginDocument{%
\def\ps@plain{\ps@empty
  \def\@oddfoot{\@monthyear \hfil \thepage}%
  \def\@evenfoot{\thepage \hfil \@volnumber}}
\def\ps@firstpage{\ps@plain}
\def\ps@headings{\ps@empty
  \def\@evenhead{%
    \setTrue{runhead}%
    \def\thanks{\protect\thanks@warning}%
    \uppercase{The Fibonacci Quarterly}\hfil}%
  \def\@oddhead{%
    \setTrue{runhead}%
    \def\thanks{\protect\thanks@warning}%
    \hfill\uppercase{Determinants of Rising Powers of Second Order Recurrence}}%
  \let\@mkboth\markboth
  \def\@evenfoot{%
    \thepage \hfil \@volnumber}%
  \def\@oddfoot{%
    \@monthyear \hfil \thepage}%
  }%
\footskip=25pt
\pagestyle{headings}%
}
\makeatother

\theoremstyle{plain}
\numberwithin{equation}{section}
\newtheorem{thm}{Theorem}[section]

\begin{document}
%% replace the values in the next three lines by the correct information
\monthyear{Month Year}
\volnumber{Volume, Number}
\setcounter{page}{1}

\title{Determinants of Rising Powers of Second Order Linear Recurrence Entries by Means of the Desnanot-Jacobi Identity}
\author{Aram Tangboonduangjit}
\address{Mahidol University International College\\
                Nakhonpathom\\
                73170, Thailand}
\email{aram.tan@mahidol.edu}
%\thanks{Research supported in part by the Natural Sciences and Engineering Research Council of Canada and by Emperor Frederick II of Sicily.}
\author{Thotsaporn Thanatipanonda}
\address{Mahidol University International College\\
                Nakhonpathom\\
                73170, Thailand}\email{thotsaporn.tha@mahidol.ac.th}

\begin{abstract}
We apply the Desnanot-Jacobi identity to give an alternative proof of the determinants whose entries are rising powers of the Fibonacci numbers given by Prodinger. We then generalize the determinants to include entries that are rising powers of the terms in a second order linear recurrence. 
\end{abstract}

\maketitle

\section{Introduction}
In 1966, Carlitz \cite{carlitz} gave the following curious formula of the determinant whose entries are powers of the Fibonacci numbers:
\begin{equation}\label{eq0} 
\abs{F_{n+i+j}^r}_{0 \leq i,j \leq r}
=(-1)^{(n+1){r+1 \choose 2}}(F_1^r F_2^{r-1}\cdots F_r)^2\cdot\prod_{i=0}^r {r\choose i}.
\end{equation}
Recently, Tangboonduangjit and Thanatipanonda \cite{AramThotsaporn} have proved this result again where the indices of the entries are slightly more general and whose method of proof is different from the one provided by Carlitz.  Another recent work which is related to the formula \eqref{eq0} is by Prodinger \cite{prodinger}. He considered the determinants whose entries are the
rising powers of the Fibonacci numbers $F_m^{\langle r \rangle}$ defined by  $$F_m^{\langle r \rangle}= F_mF_{m+1}\cdots F_{m+r-1}.$$ In particular he proved the following formula: 
\begin{equation}\label{eq1} 
\abs{F_{n+i+j}^{\langle r \rangle}}_{0 \leq i,j \leq r}
=(-1)^{n{r+1 \choose 2}+\binom{r+2}{3}}(F_1 F_2 \cdots F_r)^{r+1}.
\end{equation}
Prodinger's proof employed the LU-decomposition of the matrix whose entries are the rising powers of the Fibonacci numbers. In this work, we generalize the result of Prodinger further by making the dimension of the matrix to be independent from the rising power. This results in adding one more parameter to the formula \eqref{eq1} and we simply prove the result by induction. We then generalize the entries of the determinant to include the rising powers of the terms of a second order linear recurrence with constant coefficients. We used Maple program to facilitate some computations in this work. Thanatipanonda has included particular Maple codes associated with this work at his personal website \cite{thotsaporn}.

\section{Main Result}
We apply the Desnanot-Jacobi identity which was beautifully explained in the paper of Amdeberhan and Zeilberger \cite{dodgson} to prove the following main result:

\begin{thm} \label{main}
Let $D(n,r,d) = \abs{F_{n+i+j}^{\langle r\rangle}}_{0 \leq i,j \leq d-1}$
for integers $n,r,$ and $d$ with $r \geq 0$ and $d > 0$. Then 
$$  D(n,r,d) = (-1)^{ n \binom{d}{2} + \binom{d+1}{3}  } 
\prod_{i=1}^{d-1} (F_i F_{r+1-i})^{d-i} \cdot 
\prod_{i=d-1}^{2(d-1)} F_{n+i}^{\langle r+1-d\rangle}.$$
\end{thm}

\begin{proof}
The proof is by induction on $d$. For the base case $d=1$, we easily verify that 
$F_n^{\langle r \rangle} = D(n,r,1).$ For the case $d=2,$ we have 
\begin{align*}
\abs{F_{n+i+j}^{\langle r \rangle}}_{0 \leq i,j \leq 1}
& =  F_{n}^{\langle r \rangle}F_{n+2}^{\langle r\rangle}-F_{n+1}^{\langle r \rangle} F_{n+1}^{\langle r \rangle} \\
& =  F_{n+1}^{\langle r-1 \rangle}F_{n+2}^{\langle r-1 \rangle} (F_nF_{n+r+1}-F_{n+r}F_{n+1}) \\
& =  F_{n+1}^{\langle r-1 \rangle}F_{n+2}^{\langle r-1 \rangle} \cdot (-1)^{n+1} F_r F_1 \\
& = D(n,r,2),
\end{align*}

where we apply the well-known Vajda's identity:
\begin{equation}\label{Vajda}
F_nF_{n+i+j}- F_{n+i}F_{n+j}= (-1)^{n+1}F_iF_j\end{equation}
in the third equality. For the induction step, we assume that the result is true for all square matrices of order no greater than $d$. 
Then, by the Desnanot-Jacobi identity and the induction hypotheses, we have 
\begin{align*}
\abs{F_{n+i+j}^{\langle r\rangle}}_{0 \leq i,j \leq d} 
&= \dfrac{D(n,r,d)D(n+2,r,d) - D(n+1,r,d)^2}{ D(n+2,r,d-1)}  \\
&= \dfrac{1}{(-1)^{(n+2)\binom{d-1}{2}+\binom{d}{3}}} \cdot
 \dfrac{\left[\prod_{i=1}^{d-1}(F_i F_{r+1-i})^{d-i}\right]^2}
{\prod_{i=1}^{d-2}(F_i F_{r+1-i})^{d-1-i}}\\
&\cdot\dfrac{\prod_{i=d-1}^{2d-2}(F_{n+i}^{\langle r+1-d\rangle} 
 F_{n+2+i}^{\langle r+1-d\rangle}-F_{n+1+i}^{\langle r+1-d\rangle} F_{n+1+i}^{\langle r+1-d\rangle})}
{\prod_{i=d-2}^{2d-4} F_{n+2+i}^{\langle r+2-d\rangle}} \\
&= \dfrac{1}{(-1)^{(n+2)\binom{d-1}{2}+\binom{d}{3}}} \cdot
  F_{d-1}F_{r+2-d} \prod_{i=1}^{d-2}F_i F_{r+1-i} \cdot  \prod_{i=1}^{d-1}(F_i F_{r+1-i})^{d-i} \\
& \cdot   \dfrac{ \left[ \prod_{i=d}^{2d-2} F_{n+i}^{\langle r+1-d\rangle}F_{n+1+i}^{\langle r+1-d\rangle}      \right] \cdot
(F_{n+d-1}^{\langle r+1-d\rangle} F_{n+2d}^{\langle r+1-d\rangle}-F_{n+d}^{\langle r+1-d\rangle} F_{n+2d-1}^{\langle r+1-d\rangle})}
{\prod_{i=d}^{2d-2} F_{n+i}^{\langle r+2-d\rangle}} \\  
&= \dfrac{1}{(-1)^{(n+2)\binom{d-1}{2}+\binom{d}{3}}}
    \prod_{i=1}^{d-1}(F_i\cdot F_{r+1-i})^{d+1-i} 
  \cdot   \dfrac{  \prod_{i=d}^{2d-2} F_{n+i}^{\langle r+1-d\rangle}F_{n+1+i}^{\langle r+1-d\rangle} }
{\prod_{i=d}^{2d-2} F_{n+i}^{\langle r+2-d\rangle}} \\
& \cdot  F_{n+d}^{\langle r-d\rangle} F_{n+2d}^{\langle r-d\rangle} \cdot 
(F_{n+d-1}  F_{n+r+d}-F_{n+r} F_{n+2d-1}).
\end{align*} 

Applying the Vajda's identity \eqref{Vajda} to the last expression, we have the determinant equal

\begin{align*}
&\dfrac{1}{(-1)^{(n+2)\binom{d-1}{2}+\binom{d}{3}}}
    \prod_{i=1}^{d-1}(F_i F_{r+1-i})^{d+1-i} 
  \cdot   \dfrac{  \prod_{i=d}^{2d-2} F_{n+1+i}^{\langle r+1-d\rangle} }
{\prod_{i=d}^{2d-2} F_{n+r-d+1+i}}
\cdot  F_{n+d}^{\langle r-d\rangle} F_{n+2d}^{\langle r-d\rangle} \cdot 
 (-1)^{n+d}F_d F_{r+1-d} \\   
&= \dfrac{ (-1)^{n+d}}{(-1)^{(n+2)\binom{d-1}{2}+\binom{d}{3}}} 
 \left[F_d F_{r+1-d}\prod_{i=1}^{d-1}(F_i F_{r+1-i})^{d+1-i} \right]
    \left[\prod_{i=d}^{2d-2} F_{n+1+i}^{\langle r-d\rangle}   \right]
  F_{n+d}^{\langle r-d\rangle} F_{n+2d}^{\langle r-d\rangle}   \\  
&= (-1)^{n\binom{d+1}{2}+\binom{d+2}{3}}
  \prod_{i=1}^{d}(F_i F_{r+1-i})^{d+1-i} 
      \prod_{i=d-1}^{2d-1} F_{n+1+i}^{\langle  r-d\rangle}   \\ 
& = D(n,r,d+1).   
\end{align*}
This completes the proof by induction.
\end{proof}

Note that by letting $d=r+1$ in Theorem \ref{main} above, we obtain the original result of Prodinger, namely the identity \eqref{eq1}.

\section{Generalization to Second Order Linear Recurrence}
In this section, we let $W_n$ and  $U_n$ denote the second order linear recurrences 
with constant coefficients defined by $$W_0=a,\, 
W_1=b, \quad\text{and}\quad W_n = c_1W_{n-1}+c_2W_{n-2}
 \quad\text{for any integer $n \neq 0,1$}$$ and
 $$U_0=0,\, 
U_1=1, \quad\text{and}\quad U_n = c_1U_{n-1}+c_2U_{n-2}
 \quad\text{for any integer $n \neq 0,1$,}$$
 
 where $a, b, c_1$ and $c_2$ are any constants.
 
\begin{thm}
Let $E(n,r,d) = \abs{W_{n+i+j}^{\langle r\rangle}}_{0 \leq i,j \leq d-1}$
for integers $n,r,$ and $d$ with $r \geq 0$ and $d > 0$.
Then 
$$  E(n,r,d) = (-1)^{ n \binom{d}{2} + \binom{d+1}{3}  } 
c_2^{(n+d-2)\binom{d}{2}}  \Delta^{\binom{d}{2}} \cdot
\prod_{i=1}^{d-1} (U_i  U_{r+1-i})^{d-i} \cdot 
\prod_{i=d-1}^{2(d-1)} W_{n+i}^{\langle r+1-d\rangle},$$
where $\Delta = \begin{vmatrix}
  W_{1} & W_{2}  \\
  W_{0} & W_{1} \\
 \end{vmatrix} = b^2-c_1ab-c_2a^2$.  
\end{thm}

\begin{proof}
The proof is similar to that of Theorem \ref{main}. However, instead of using the Vajda's identity, we use the following identity: \begin{equation} \label{eqWW}
W_nW_{n+i+j}-W_{n+i}W_{n+j} =  (-1) \cdot (-c_2)^{n} \cdot \Delta \cdot U_iU_j.
\end{equation}
which can be found in  the recent paper by Tangboonduangjit and Thanatipanonda \cite{AramThotsaporn}.
\end{proof}

\medskip

\noindent MSC2010: 11B39, 33C05

\end{document}